\documentclass{amsart}
\usepackage[dvips]{graphicx}
\usepackage{amssymb}

\newtheorem{theorem}{Theorem}[section]

\newtheorem{lemma}[theorem]{Lemma}
\newtheorem{proposition}[theorem]{Proposition}

\theoremstyle{remark}
\newtheorem{definition}[theorem]{\sc Definition}

\newtheorem{example}[theorem]{\sc Example}
\newtheorem*{acknowledgment}{\sc Acknowledgment}

\numberwithin{equation}{section}

\begin{document}

\title{Toric Fano varieties associated to building sets}

\author{Yusuke Suyama}
\address{Department of Mathematics, Graduate School of Science, Osaka City University,
3-3-138 Sugimoto, Sumiyoshi-ku, Osaka 558-8585 JAPAN}
\email{d15san0w03@st.osaka-cu.ac.jp}



\date{\today}


\begin{abstract}
We characterize building sets
whose associated nonsingular projective toric varieties are Fano.
Furthermore, we show that all such toric Fano varieties
are obtained from smooth Fano polytopes associated to finite directed graphs.
\end{abstract}

\maketitle

\section{Introduction}

An $n$-dimensional {\it toric variety} is a normal algebraic variety $X$ over $\mathbb{C}$
containing the algebraic torus $(\mathbb{C}^*)^n$ as an open dense subset,
such that the natural action of $(\mathbb{C}^*)^n$ on itself extends to an action on $X$.
The category of toric varieties is equivalent to the category of fans,
which are combinatorial objects.

A nonsingular projective algebraic variety is said to be {\it Fano}
if its anticanonical divisor is ample.
The classification of toric Fano varieties is a fundamental problem
and has been studied by many researchers.
In particular, {\O}bro \cite{Obro} gave an algorithm that classifies all toric Fano varieties
for any given dimension.

There is a construction of nonsingular projective toric varieties from building sets.
The class of such toric varieties includes toric varieties corresponding to
graph associahedra of finite simple graphs \cite{Postnikov}.
On the other hand, Higashitani \cite{Higashitani}
gave a construction of integral convex polytopes from finite directed graphs.
There is a one-to-one correspondence between smooth Fano polytopes
and toric Fano varieties.
He also gave a necessary and sufficient condition for
the polytope to be smooth Fano in terms of the finite directed graph.

In this paper, we give a necessary and sufficient condition
for the toric variety associated to a building set to be Fano in terms of the building set
(Theorem \ref{theorem1}).
The author \cite{Suyama} characterized finite simple graphs
whose associated toric varieties are Fano.
Theorem \ref{theorem1} generalizes this result (Example \ref{examples} (2)).
Furthermore, we prove that any toric Fano variety associated to a building set
is obtained from the smooth Fano polytope
associated to a finite directed graph (Theorem \ref{theorem2}).

The structure of the paper is as follows.
In Section 2, we state the characterization of building sets
whose associated toric varieties are Fano.
In Section 3, we give its proof.
In Section 4, we show that all such toric Fano varieties
are obtained from finite directed graphs.

\begin{acknowledgment}
This work was supported by Grant-in-Aid for JSPS Fellows 15J01000.
The author wishes to thank his supervisor, Professor Mikiya Masuda,
for his continuing support.
Professor Akihiro Higashitani gave me valuable suggestions and comments.
\end{acknowledgment}

\section{Building sets whose associated toric varieties are Fano}

We review the construction of a toric variety from a building set.
Let $S$ be a nonempty finite set.
A {\it building set} on $S$ is a finite set $B$ of nonempty subsets of $S$
satisfying the following conditions:
\begin{enumerate}
\item If $I, J \in B$ and $I \cap J \ne \emptyset$, then we have $I \cup J \in B$.
\item For every $i \in S$, we have $\{i\} \in B$.
\end{enumerate}
We denote by $B_{\rm max}$ the set of all maximal (by inclusion) elements of $B$.
An element of $B_{\rm max}$ is called a {\it $B$-component}
and $B$ is said to be {\it connected} if $B_{\rm max}=\{S\}$.
For a nonempty subset $C$ of $S$, we call $B|_C=\{I \in B \mid I \subset C\}$
the {\it restriction} of $B$ to $C$. $B|_C$ is a building set on $C$.
Note that we have $B=\bigsqcup_{C \in B_{\rm max}} B|_C$ for any building set $B$.
In particular, any building set is a disjoint union of connected building sets.

\begin{definition}\label{nested}
A {\it nested set} of $B$ is a subset $N$ of $B \setminus B_{\rm max}$
satisfying the following conditions:
\begin{enumerate}
\item If $I, J \in N$, then we have either
$I \subset J$ or $J \subset I$ or $I \cap J=\emptyset$.
\item For any integer $k \geq 2$ and for any pairwise disjoint $I_1, \ldots, I_k \in N$,
we have $I_1 \cup \cdots \cup I_k \notin B$.
\end{enumerate}
\end{definition}

The set $\mathcal{N}(B)$ of all nested sets of $B$ is called the {\it nested complex}.
$\mathcal{N}(B)$ is a simplicial complex on $B \setminus B_{\rm max}$.

\begin{proposition}[{\cite[Proposition 4.1]{Zelevinsky}}]\label{pure}
Let $B$ be a building set on $S$.
Then all maximal (by inclusion) nested sets of $B$
have the same cardinality $|S|-|B_{\rm max}|$.
In particular, if $B$ is connected,
then the cardinality of a maximal nested set of $B$ is $|S|-1$.
\end{proposition}

First, suppose that $B$ is a connected building set on $S$.
Let $S=\{1, \ldots, n+1\}$.
We denote by $e_1, \ldots, e_n$ the standard basis for $\mathbb{R}^n$
and we put $e_{n+1}=-e_1-\cdots-e_n$.
For $I \subset S$, we denote $e_I=\sum_{i \in I}e_i$.
For $N \in \mathcal{N}(B)$, we denote by $\mathbb{R}_{\geq 0}N$
the $|N|$-dimensional cone $\sum_{I \in N}\mathbb{R}_{\geq 0}e_I$,
where $\mathbb{R}_{\geq 0}$ is the set of non-negative real numbers.
We define $\Delta(B)=\{\mathbb{R}_{\geq 0}N \mid N \in \mathcal{N}(B)\}$.
Then $\Delta(B)$ is a fan in $\mathbb{R}^n$
and thus we have an $n$-dimensional toric variety $X(\Delta(B))$.
If $B$ is not connected,
then we define $X(\Delta(B))=\prod_{C \in B_{\rm max}}X(\Delta(B|_C))$.

\begin{theorem}[{\cite[Corollary 5.2 and Theorem 6.1]{Zelevinsky}}]
Let $B$ be a building set.
Then the associated toric variety $X(\Delta(B))$ is nonsingular and projective.
\end{theorem}

\begin{example}
Let $S=\{1, 2, 3\}$ and $B=\{\{1\}, \{2\}, \{3\}, \{2, 3\}, \{1, 2, 3\}\}$.
Then the nested complex $\mathcal{N}(B)$ is
\begin{align*}
&\{\emptyset, \{\{1\}\}, \{\{2\}\}, \{\{3\}\}, \{\{2, 3\}\},\\
&\{\{1\}, \{2\}\}, \{\{1\}, \{3\}\}, \{\{2\}, \{2, 3\}\}, \{\{3\}, \{2, 3\}\}\}.
\end{align*}
Hence we have the fan $\Delta(B)$ in Figure \ref{example}.
Therefore the corresponding toric variety $X(\Delta(B))$ is $\mathbb{P}^2$
blown-up at one point.
\begin{figure}[htbp]
\begin{center}
\includegraphics[width=5cm]{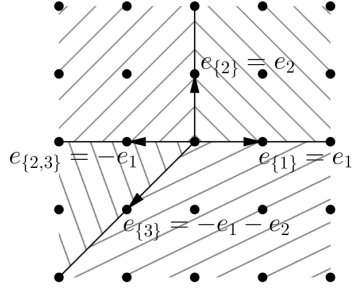}
\caption{the fan $\Delta(B)$.}
\label{example}
\end{center}
\end{figure}
\end{example}

Our first main result is the following:

\begin{theorem}\label{theorem1}
Let $B$ be a building set. Then the following are equivalent:
\begin{enumerate}
\item The associated nonsingular projective toric variety $X(\Delta(B))$ is Fano.
\item For any $B$-component $C$ and
for any $I_1, I_2 \in B|_C$ such that $I_1 \cap I_2 \ne \emptyset,
I_1 \not\subset I_2$ and $I_2 \not\subset I_1$,
we have $I_1 \cup I_2=C$ and $I_1 \cap I_2 \in B|_C$.
\end{enumerate}
\end{theorem}

\begin{example}\label{examples}
\begin{enumerate}
\item If $|S| \leq 3$, then a connected building set $B$ on $S$
is isomorphic to one of the following six types:
\begin{enumerate}
\item $\{\{1\}\}$: a point, which is understood to be Fano.
\item $\{\{1\}, \{2\}, \{1, 2\}\}$: $\mathbb{P}^1$.
\item $\{\{1\}, \{2\}, \{3\}, \{1, 2, 3\}\}$: $\mathbb{P}^2$.
\item $\{\{1\}, \{2\}, \{3\}, \{1, 2\}, \{1, 2, 3\}\}$: $\mathbb{P}^2$ blown-up at one point.
\item $\{\{1\}, \{2\}, \{3\}, \{1, 2\}, \{1, 3\}, \{1, 2, 3\}\}$:
$\mathbb{P}^2$ blown-up at two points.
\item $\{\{1\}, \{2\}, \{3\}, \{1, 2\}, \{1, 3\}, \{2, 3\}, \{1, 2, 3\}\}$:
$\mathbb{P}^2$ blown-up at three points.
\end{enumerate}
Thus $X(\Delta(B))$ is Fano in every case.
Since the disconnected building set $\{\{1\}, \{2\}, \{1, 2\}, \{3\}, \{4\}, \{3, 4\}\}$
yields $\mathbb{P}^1 \times \mathbb{P}^1$,
it follows that all toric Fano varieties of dimension $\leq 2$ are obtained from building sets.
\item Let $G$ be a finite simple graph, that is,
a finite graph with no loops and no multiple edges.
We denote by $V(G)$ and $E(G)$ its node set and edge set respectively.
For $I \subset V(G)$, we define a graph $G|_I$
by $V(G|_I)=I$ and $E(G|_I)=\{\{v, w\} \in E(G) \mid v, w \in I\}$.
The {\it graphical building set} $B(G)$ of $G$ is defined to be
$\{I \subset V(G) \mid G|_I \mbox{ is connected}, I \ne \emptyset\}$.
Theorem \ref{theorem1} implies that
the toric variety $X(\Delta(B(G)))$ is Fano if and only if
each connected component of $G$ has at most three nodes,
which agrees with \cite[Theorem 3.1]{Suyama}.
\item If $|S|=4$, then a connected building set $B$ on $S$
whose associated toric variety is Fano
is isomorphic to one of the following nine types:
\begin{enumerate}
\item \{\{1\}, \{2\}, \{3\}, \{4\}, \{1, 2, 3, 4\}\}.
\item \{\{1\}, \{2\}, \{3\}, \{4\}, \{1, 2, 3\}, \{1, 2, 3, 4\}\}.
\item \{\{1\}, \{2\}, \{3\}, \{4\}, \{1, 2\}, \{1, 2, 3, 4\}\}.
\item \{\{1\}, \{2\}, \{3\}, \{4\}, \{1, 2\}, \{3, 4\}, \{1, 2, 3, 4\}\}.
\item \{\{1\}, \{2\}, \{3\}, \{4\}, \{1, 2\}, \{1, 2, 3\}, \{1, 2, 3, 4\}\}.
\item \{\{1\}, \{2\}, \{3\}, \{4\}, \{3, 4\}, \{1, 2, 3\}, \{1, 2, 3, 4\}\}.
\item \{\{1\}, \{2\}, \{3\}, \{4\}, \{1, 2\}, \{3, 4\}, \{1, 2, 3\}, \{1, 2, 3, 4\}\}.
\item \{\{1\}, \{2\}, \{3\}, \{4\}, \{1, 2\}, \{1, 2, 3\}, \{1, 2, 4\}, \{1, 2, 3, 4\}\}.
\item \{\{1\}, \{2\}, \{3\}, \{4\}, \{1, 2\}, \{3, 4\}, \{1, 2, 3\}, \{1, 2, 4\}, \{1, 2, 3, 4\}\}.
\end{enumerate}
Among 18 types of toric Fano threefolds,
13 types are indecomposable and five types are products of
$\mathbb{P}^1$ and toric del Pezzo surfaces (see, for example \cite[pp.90--92]{Oda}).
This shows that there are nine types of indecomposable toric Fano threefolds
that are obtained from building sets.
On the other hand, (1) shows that all toric del Pezzo surfaces
are obtained from building sets.
Thus there are exactly 14 types of toric Fano threefolds
that are obtained from building sets.
\end{enumerate}
\end{example}

\section{Proof of Theorem \ref{theorem1}}

We recall a description of the intersection number of the anticanonical divisor
with a torus-invariant curve, see \cite{Oda} for details.
For a nonsingular complete fan $\Delta$ in $\mathbb{R}^n$ and $0 \leq r \leq n$,
we denote by $\Delta(r)$ the set of $r$-dimensional cones of $\Delta$.
We denote by $X(\Delta)$ the associated toric variety.
For $\tau \in \Delta(n-1)$,
the intersection number of the anticanonical divisor $-K_{X(\Delta)}$ with
the torus-invariant curve $V(\tau)$ corresponding to $\tau$
can be computed as follows:

\begin{proposition}\label{intersectionnumber}
Let $X(\Delta)$ be an $n$-dimensional nonsingular complete toric variety
and $\tau=\mathbb{R}_{\geq 0}v_1+\cdots+\mathbb{R}_{\geq 0}v_{n-1} \in \Delta(n-1)$,
where $v_1, \ldots, v_{n-1}$ are primitive vectors in $\mathbb{Z}^n$.
Let $v$ and $v'$ be the distinct primitive vectors in $\mathbb{Z}^n$ such that
$\tau+\mathbb{R}_{\geq 0}v$ and $\tau+\mathbb{R}_{\geq 0}v'$ are in $\Delta(n)$.
Then there exist integers $a_1, \ldots, a_{n-1}$ such that
$v+v'+a_1v_1+\cdots+a_{n-1}v_{n-1}=0$.
The intersection number $(-K_{X(\Delta)}.V(\tau))$ is equal to
$2+a_1+\cdots+a_{n-1}$.
\end{proposition}

\begin{proposition}\label{Fano}
Let $X(\Delta)$ be an $n$-dimensional nonsingular complete toric variety.
Then $X(\Delta)$ is Fano if and only if
$(-K_{X(\Delta)}.V(\tau))$ is positive for every $\tau \in \Delta(n-1)$.
\end{proposition}

Let $B$ be a building set on $S$.
For $C \in B \setminus B_{\rm max}$, we call
\begin{equation*}
\mathcal{N}(B)_C=\{N \subset (B \setminus B_{\rm max}) \setminus \{C\} \mid
N \cup \{C\} \in \mathcal{N}(B)\}
\end{equation*}
the {\it link} of $C$ in $\mathcal{N}(B)$. $\mathcal{N}(B)_C$ is a simplicial complex on
\begin{equation*}
\{I \in (B \setminus B_{\rm max}) \setminus \{C\} \mid
\{I, C\} \in \mathcal{N}(B)\}.
\end{equation*}
For a nonempty proper subset $C$ of $S$, we call
\begin{equation*}
C \setminus B=\{I \subset S \setminus C \mid I \ne \emptyset;
I \in B \mbox{ or } C \cup I \in B\}
\end{equation*}
the {\it contraction} of $C$ from $B$.
$C \setminus B$ is a building set on $S \setminus C$.

\begin{proposition}[{\cite[Proposition 3.2]{Zelevinsky}}]\label{link}
Let $B$ be a building set on $S$ and let $C \in B \setminus B_{\rm max}$.
Then the correspondence
\begin{equation*}
I \mapsto \left\{\begin{array}{ll}
I \setminus C & (C \subset I), \\
I & (C \not\subset I) \end{array}\right.
\end{equation*}
induces an isomorphism
$\mathcal{N}(B)_C \rightarrow \mathcal{N}(B|_C \cup (C \setminus B))$
of simplicial complexes.
\end{proposition}

The {\it symmetric difference} of two sets $X$ and $Y$
is defined by $X \triangle Y=(X \cup Y) \setminus (X \cap Y)$.
The following is the key lemma.

\begin{lemma}\label{keylemma}
Let $B$ be a connected building set on $S$
and let $I_1, I_2 \in B$ with $I_1 \cap I_2 \ne \emptyset, I_1 \not\subset I_2$
and $I_2 \not\subset I_1$. Then the following hold:
\begin{enumerate}
\item There exist $J_1, J_2 \in B$ with
$J_1 \cap J_2 \ne \emptyset$ and $J_1 \cup J_2 \subset I_1 \cup I_2$,
$j_1 \in J_1 \setminus J_2, j_2 \in J_2 \setminus J_1$,
a maximal nested set $N$ of $B|_{J_1 \cap J_2}$ and 
a maximal nested set $N'$ of $B|_{(J_1 \triangle J_2) \setminus \{j_1, j_2\}}$ such that
\begin{equation}\label{J_k}
\{J_k\} \cup N \cup (B|_{J_1 \cap J_2})_{\rm max} \cup N' \cup
(B|_{(J_1 \triangle J_2) \setminus \{j_1, j_2\}})_{\rm max}
\end{equation}
are nested sets of $B$ for $k=1, 2$.
If $I_1 \cap I_2 \notin B$, then we can choose $J_1, J_2 \in B$
so that $J_1 \cap J_2 \notin B$ or $J_1 \cup J_2 \subsetneq I_1 \cup I_2$.
\item Furthermore, if $J_1 \cup J_2 \subsetneq S$,
then there exists a nested set $N''$ of $B$ such that
\begin{equation*}
\{J_k, J_1 \cup J_2\} \cup N \cup (B|_{J_1 \cap J_2})_{\rm max} \cup N' \cup
(B|_{(J_1 \triangle J_2) \setminus \{j_1, j_2\}})_{\rm max} \cup N''
\end{equation*}
are maximal nested sets of $B$ for $k=1, 2$ ($N''$ can be empty).
\end{enumerate}
If $J_1 \triangle J_2=\{j_1, j_2\}$,
then $N'$ and $(B|_{(J_1 \triangle J_2) \setminus \{j_1, j_2\}})_{\rm max}$
are understood to be empty.
\end{lemma}

\begin{proof}
(1) We use induction on $|I_1 \triangle I_2|$.
We have $|I_1 \triangle I_2| \geq 2$.
Suppose $|I_1 \triangle I_2|=2$. We put $J_1=I_1$ and $J_2=I_2$.
Clearly $J_1 \cap J_2 \ne \emptyset$ and $J_1 \cup J_2 \subset I_1 \cup I_2$.
We choose any maximal nested set $N$ of $B|_{J_1 \cap J_2}$. Then
$\{J_1\} \cup N \cup (B|_{J_1 \cap J_2})_{\rm max}$ and
$\{J_2\} \cup N \cup (B|_{J_1 \cap J_2})_{\rm max}$ are nested sets of $B$.
If $I_1 \cap I_2 \notin B$, then $J_1 \cap J_2 \notin B$.

Suppose $|I_1 \triangle I_2| \geq 3$.
We choose $i_1 \in I_1 \setminus I_2, i_2 \in I_2 \setminus I_1$,
and maximal nested sets $N$ and $N'$ of $B|_{I_1 \cap I_2}$ and 
$B|_{(I_1 \triangle I_2) \setminus \{i_1, i_2\}}$, respectively. If
\begin{equation*}
\{I_k\} \cup N \cup (B|_{I_1 \cap I_2})_{\rm max} \cup N' \cup
(B|_{(I_1 \triangle I_2) \setminus \{i_1, i_2\}})_{\rm max}
\end{equation*}
are nested sets of $B$ for $k=1, 2$, then there is nothing to prove.
Without loss of generality, we may assume that
\begin{equation}\label{I_k}
\{I_1\} \cup N \cup (B|_{I_1 \cap I_2})_{\rm max} \cup N' \cup
(B|_{(I_1 \triangle I_2) \setminus \{i_1, i_2\}})_{\rm max}
\end{equation}
is not a nested set of $B$.
We find $I'_1, I'_2 \in B$ satisfying $I_1 \cap I_2 \subsetneq I'_1 \cap I'_2,
I'_1 \not\subset I'_2, I'_2 \not\subset I'_1$ and $I'_1 \cup I'_2=I_1 \cup I_2$ as follows:

{\it Case 1}. Suppose that (\ref{I_k})
does not satisfy the condition (1) in Definition \ref{nested}.
$\{I_1\} \cup N \cup (B|_{I_1 \cap I_2})_{\rm max}$ and
$N' \cup (B|_{(I_1 \triangle I_2) \setminus \{i_1, i_2\}})_{\rm max}$ are nested sets.
For any $K \in N \cup (B|_{I_1 \cap I_2})_{\rm max}$
and $L \in N' \cup (B|_{(I_1 \triangle I_2) \setminus \{i_1, i_2\}})_{\rm max}$,
we have $K \cap L=\emptyset$. Hence there exists
$L \in N' \cup (B|_{(I_1 \triangle I_2) \setminus \{i_1, i_2\}})_{\rm max}$
such that $I_1 \not\subset L, L \not\subset I_1$ and $I_1 \cap L \ne \emptyset$.
Then $I_1 \cup L \in B$. We put $I'_1=I_1 \cup L$ and $I'_2=I_2$.
Since $L \subset I_1 \triangle I_2$,
it follows that $L \setminus I_1 \subset (I'_1 \cap I'_2) \setminus (I_1 \cap I_2)$.
Thus $I_1 \cap I_2 \subsetneq I'_1 \cap I'_2$.

{\it Case 2}. Suppose that (\ref{I_k})
does not satisfy the condition (2) in Definition \ref{nested},
and there exist
\begin{equation*}
K_1, \ldots, K_r \in N \cup (B|_{I_1 \cap I_2})_{\rm max},\quad
L_1, \ldots, L_s \in N' \cup (B|_{(I_1 \triangle I_2) \setminus \{i_1, i_2\}})_{\rm max}
\end{equation*}
for $r, s \geq 1$ such that $K_1, \ldots, K_r, L_1, \ldots, L_s$ are pairwise disjoint
and $K_1 \cup \cdots \cup K_r \cup L_1 \cup \cdots \cup L_s \in B$.
Then we have $I_k \cup L_1 \cup \cdots \cup L_s \in B$ for $k=1, 2$.
We put $I'_k=I_k \cup L_1 \cup \cdots \cup L_s$ for $k=1, 2$.
Since $L_1 \cup \cdots \cup L_s \subset I_1 \triangle I_2$,
we must have $I_1 \subsetneq I'_1$ or $I_2 \subsetneq I'_2$.
If $I_1 \subsetneq I'_1$, then it follows that
$I'_1 \setminus I_1 \subset (I'_1 \cap I'_2) \setminus (I_1 \cap I_2)$.
Thus $I_1 \cap I_2 \subsetneq I'_1 \cap I'_2$.
Similarly, $I_2 \subsetneq I'_2$ implies $I_1 \cap I_2 \subsetneq I'_1 \cap I'_2$.

{\it Case 3}. Suppose that (\ref{I_k})
does not satisfy the condition (2) in Definition \ref{nested},
and there exist
$L_1, \ldots, L_s \in N' \cup (B|_{(I_1 \triangle I_2) \setminus \{i_1, i_2\}})_{\rm max}$
such that $I_1, L_1, \ldots, L_s$ are pairwise disjoint
and $I_1 \cup L_1 \cup \cdots \cup L_s \in B$.
We put $I'_1=I_1 \cup L_1 \cup \cdots \cup L_s$ and $I'_2=I_2$.
Since $L_1 \cup \cdots \cup L_s \subset I_1 \triangle I_2$,
it follows that
$(L_1 \cup \cdots \cup L_s) \setminus I_1 \subset (I'_1 \cap I'_2) \setminus (I_1 \cap I_2)$.
Thus $I_1 \cap I_2 \subsetneq I'_1 \cap I'_2$.

In every case, we have $i_1 \in I'_1 \setminus I'_2, i_2 \in I'_2 \setminus I'_1$
and $I'_1 \cup I'_2=I_1 \cup I_2$.
Hence $|I'_1 \triangle I'_2|=|I'_1 \cup I'_2|-|I'_1 \cap I'_2|
<|I_1 \cup I_2|-|I_1 \cap I_2|=|I_1 \triangle I_2|$.
By the hypothesis of induction, there exist $J_1, J_2 \in B$  with
$J_1 \cap J_2 \ne \emptyset$ and $J_1 \cup J_2 \subset I'_1 \cup I'_2=I_1 \cup I_2$,
$j_1 \in J_1 \setminus J_2, j_2 \in J_2 \setminus J_1$,
a maximal nested set $N$ of $B|_{J_1 \cap J_2}$ and 
a maximal nested set $N'$ of $B|_{(J_1 \triangle J_2) \setminus \{j_1, j_2\}}$ such that
(\ref{J_k}) are nested sets of $B$ for $k=1, 2$.

Suppose that $I_1 \cap I_2 \notin B$.
If $I'_1 \cap I'_2 \notin B$, then by the hypothesis of induction,
we have $J_1 \cap J_2 \notin B$ or $J_1 \cup J_2 \subsetneq I'_1 \cup I'_2=I_1 \cup I_2$.
Suppose $I'_1 \cap I'_2 \in B$. We may assume that $I_1 \subsetneq I'_1$.
We put $I''_1=I_1$ and $I''_2=I'_1 \cap I'_2$.
We have $I''_1 \cap I''_2=I_1 \cap I'_2 \supset I_1 \cap I_2 \ne \emptyset$
and $I''_1 \cup I''_2 \subset I_1 \cup I_2$.
Since $I''_2 \subset I'_1$ and $i_2 \notin I'_1$,
it follows that $i_2 \in (I_1 \cup I_2) \setminus (I''_1 \cup I''_2)$.
Hence $|I''_1 \triangle I''_2|=|I''_1 \cup I''_2|-|I''_1 \cap I''_2|
<|I_1 \cup I_2|-|I_1 \cap I_2|=|I_1 \triangle I_2|$.
We have $i_1 \in I''_1 \setminus I''_2$ and $I'_1 \setminus I_1 \subset I''_2 \setminus I''_1$,
since $I'_1 \cup I'_2=I_1 \cup I_2$.
By the hypothesis of induction, there exist $J_1, J_2 \in B$  with
$J_1 \cap J_2 \ne \emptyset$ and
$J_1 \cup J_2 \subset I''_1 \cup I''_2 \subsetneq I_1 \cup I_2$,
$j_1 \in J_1 \setminus J_2, j_2 \in J_2 \setminus J_1$,
a maximal nested set $N$ of $B|_{J_1 \cap J_2}$ and 
a maximal nested set $N'$ of $B|_{(J_1 \triangle J_2) \setminus \{j_1, j_2\}}$ such that
(\ref{J_k}) are nested sets of $B$ for $k=1, 2$.

Therefore the assertion holds for $|I_1 \triangle I_2|$.

(2) We see that
\begin{align*}
&|\{J_k\} \cup N \cup (B|_{J_1 \cap J_2})_{\rm max} \cup N' \cup
(B|_{(J_1 \triangle J_2) \setminus \{j_1, j_2\}})_{\rm max}|\\
&=1+|J_1 \cap J_2|+|(J_1 \triangle J_2) \setminus \{j_1, j_2\}|\\
&=|J_1 \cup J_2|-1
\end{align*}
for $k=1,2$. Hence by Proposition \ref{pure},
(\ref{J_k}) are maximal nested sets of $B|_{J_1 \cup J_2}$.
We choose any maximal nested set $M$ of $(J_1 \cup J_2) \setminus B$. Then
\begin{equation*}
\{J_k\} \cup N \cup (B|_{J_1 \cap J_2})_{\rm max} \cup N' \cup
(B|_{(J_1 \triangle J_2) \setminus \{j_1, j_2\}})_{\rm max} \cup M
\end{equation*}
are maximal nested sets of $B|_{J_1 \cup J_2} \cup ((J_1 \cup J_2) \setminus B)$.
By Proposition \ref{link},
\begin{equation*}
\{J_k\} \cup N \cup (B|_{J_1 \cap J_2})_{\rm max} \cup N' \cup
(B|_{(J_1 \triangle J_2) \setminus \{j_1, j_2\}})_{\rm max} \cup N''
\end{equation*}
are in $\mathcal{N}(B)_{J_1 \cup J_2}$ for some $N'' \in \mathcal{N}(B)$. Thus
\begin{equation*}
\{J_k, J_1 \cup J_2\} \cup N \cup (B|_{J_1 \cap J_2})_{\rm max} \cup N' \cup
(B|_{(J_1 \triangle J_2) \setminus \{j_1, j_2\}})_{\rm max} \cup N''
\end{equation*}
are maximal nested sets of $B$.
\end{proof}

\begin{example}
The proof of Lemma \ref{keylemma} (1) gives a method for obtaining explicit $J_1$ and $J_2$.
Let $S=\{1, 2, 3, 4, 5, 6\}$,
\begin{align*}
B&=\{\{1\}, \{2\}, \{3\}, \{4\}, \{5\}, \{6\}, \{2, 5\}, \{2, 3, 4\}, \{3, 4, 5\}, \{1, 2, 3, 4\},\\
&\{2, 3, 4, 5\}, \{3, 4, 5, 6\}, \{1, 2, 3 ,4, 5\}, \{2, 3, 4, 5, 6\}, \{1, 2, 3, 4, 5, 6\}\},
\end{align*}
$I_1=\{1, 2, 3, 4\}, I_2=\{3, 4, 5, 6\}, i_1=1$ and $i_2=6$.
Then $I_1 \cap I_2=\{3, 4\} \notin B$.
We have
\begin{equation*}
B|_{I_1 \cap I_2}=\{\{3\}, \{4\}\},\quad
B|_{(I_1 \triangle I_2) \setminus \{i_1, i_2\}}=\{\{2\}, \{5\}, \{2, 5\}\}.
\end{equation*}
$\emptyset$ and $\{\{2\}\}$ are maximal nested sets of
$B|_{I_1 \cap I_2}$ and $B|_{(I_1 \triangle I_2) \setminus \{i_1, i_2\}}$, respectively.
However,
\begin{align*}
&\{I_1\} \cup \emptyset \cup (B|_{I_1 \cap I_2})_{\rm max} \cup \{\{2\}\}
\cup(B|_{(I_1 \triangle I_2) \setminus \{i_1, i_2\}})_{\rm max}\\
&=\{\{1, 2, 3, 4\}, \{3\}, \{4\}, \{2\}, \{2, 5\}\}
\end{align*}
is not a nested set because of $I_1=\{1, 2, 3, 4\}$ and $L=\{2, 5\}$ (Case 1).
Thus we put $I'_1=I_1 \cup L=\{1, 2, 3, 4, 5\}$ and $I'_2=I_2=\{3, 4, 5, 6\}$.
But $I'_1 \cap I'_2=\{3, 4, 5\} \in B$.
Thus we put $I''_1=I_1=\{1, 2, 3, 4\}, I''_2=I'_1 \cap I'_2=\{3, 4, 5\},
i''_1=1$ and $i''_2=5$.
Then we have
\begin{equation*}
B|_{I''_1 \cap I''_2}=\{\{3\}, \{4\}\},\quad
B|_{(I''_1 \triangle I''_2) \setminus \{i''_1, i''_2\}}=\{\{2\}\}.
\end{equation*}
The only maximal nested set of each is the empty set. However,
\begin{align*}
&\{I''_1\} \cup \emptyset \cup (B|_{I''_1 \cap I''_2})_{\rm max} \cup \emptyset
\cup(B|_{(I''_1 \triangle I''_2) \setminus \{i''_1, i''_2\}})_{\rm max}\\
&=\{\{1, 2, 3, 4\}, \{3\}, \{4\}, \{2\}\}
\end{align*}
is not a nested set because $\{2, 3, 4\} \in B$ (Case 2).
Thus we put $J_1=I''_1 \cup \{2\}=\{1, 2, 3, 4\},
J_2=I''_2 \cup \{2\}=\{2, 3, 4, 5\}, j_1=1$ and $j_2=5$.
Then we have
\begin{equation*}
B|_{J_1 \cap J_2}=\{\{2\}, \{3\}, \{4\}, \{2, 3, 4\}\},\quad
J_1 \triangle J_2=\{j_1, j_2\}.
\end{equation*}
We choose $\{\{2\}, \{3\}\}$ as a maximal nested set of $B|_{J_1 \cap J_2}$. Then
\begin{align*}
\{J_1\} \cup \{\{2\}, \{3\}\} \cup (B|_{J_1 \cap J_2})_{\rm max}
&=\{\{1, 2, 3, 4\}, \{2\}, \{3\}, \{2, 3, 4\}\},\\
\{J_2\} \cup \{\{2\}, \{3\}\} \cup (B|_{J_1 \cap J_2})_{\rm max}
&=\{\{2, 3, 4, 5\}, \{2\}, \{3\}, \{2, 3, 4\}\}
\end{align*}
are nested sets of $B$.
\end{example}

\begin{proposition}[{\cite[Proposition 4.5]{Zelevinsky}}]\label{pair}
Let $B$ be a building set on $S$
and let $N \cup \{I_1\}$ and $N \cup \{I_2\}$ be two maximal nested sets of $B$
with the intersection $N \in \mathcal{N}(B)$.
Then the following hold:
\begin{enumerate}
\item We have $I_1 \not\subset I_2$ and $I_2 \not\subset I_1$.
\item If $I_1 \cap I_2 \ne \emptyset$, then $(B|_{I_1 \cap I_2})_{\rm max} \subset N$.
\item There exist $I_3, \ldots, I_k \in N$ such that
$I_1 \cup I_2, I_3, \ldots, I_k$ are pairwise disjoint
and $I_1 \cup \cdots \cup I_k \in N \cup B_{\rm max}$ ($\{I_3, \ldots, I_k\}$ can be empty).
\end{enumerate}
\end{proposition}

\begin{proof}[Proof of Theorem \ref{theorem1}]
Any building set is a disjoint union of connected building sets.
The disjoint union of connected building sets corresponds
to the product of toric varieties associated to the connected building sets.
The product of nonsingular projective toric varieties is Fano if and only if
every factor is Fano. Hence it suffices to show that,
for any connected building set $B$ on $S=\{1, \ldots, n+1\}$,
the following are equivalent:
\begin{itemize}
\item[($1'$)] $X(\Delta(B))$ is Fano.
\item[($2'$)] $I_1, I_2 \in B, I_1 \cap I_2 \ne \emptyset, I_1 \not\subset I_2, I_2 \not\subset I_1
\Rightarrow I_1 \cup I_2=S \mbox{ and } I_1 \cap I_2 \in B$.
\end{itemize}

$(1') \Rightarrow (2')$: Suppose that there exist $I_1, I_2 \in B$ with
$I_1 \cap I_2 \ne \emptyset, I_1 \not\subset I_2, I_2 \not\subset I_1$
such that $I_1 \cup I_2 \subsetneq S \mbox{ or } I_1 \cap I_2 \notin B$.
We will use the notation of Lemma \ref{keylemma}.

{\it The case where $I_1 \cup I_2 \subsetneq S$}.
By Lemma \ref{keylemma}, we have maximal nested sets
\begin{equation*}
\{J_k, J_1 \cup J_2\} \cup N \cup (B|_{J_1 \cap J_2})_{\rm max} \cup N' \cup
(B|_{(J_1 \triangle J_2) \setminus \{j_1, j_2\}})_{\rm max} \cup N''
\end{equation*}
of $B$ for $k=1, 2$. Let
\begin{equation*}
\tau=\mathbb{R}_{\geq0}(\{J_1 \cup J_2\} \cup N \cup (B|_{J_1 \cap J_2})_{\rm max}
\cup N' \cup (B|_{(J_1 \triangle J_2) \setminus \{j_1, j_2\}})_{\rm max} \cup N'').
\end{equation*}
Clearly
\begin{equation*}
e_{J_1}+e_{J_2}-\sum_{C \in (B|_{J_1 \cap J_2})_{\rm max}}e_C-e_{J_1 \cup J_2}=0.
\end{equation*}
Hence by Proposition \ref{intersectionnumber},
we have $(-K_{X(\Delta(B))}.V(\tau))=2-|(B|_{J_1 \cap J_2})_{\rm max}|-1 \leq 0$.
By Proposition \ref{Fano}, $X(\Delta(B))$ is not Fano.

{\it The case where $I_1 \cup I_2=S$ and $I_1 \cap I_2 \notin B$}.
By Lemma \ref{keylemma} (1), we have nested sets
\begin{equation*}
\{J_k\} \cup N \cup (B|_{J_1 \cap J_2})_{\rm max} \cup N' \cup
(B|_{(J_1 \triangle J_2) \setminus \{j_1, j_2\}})_{\rm max}
\end{equation*}
of $B$ for $k=1, 2$,
where $J_1 \cap J_2 \notin B$ or $J_1 \cup J_2 \subsetneq I_1 \cup I_2=S$.
If $J_1 \cup J_2 \subsetneq S$,
then by Lemma \ref{keylemma} (2), we have maximal nested sets
\begin{equation*}
\{J_k, J_1 \cup J_2\} \cup N \cup (B|_{J_1 \cap J_2})_{\rm max} \cup N' \cup
(B|_{(J_1 \triangle J_2) \setminus \{j_1, j_2\}})_{\rm max} \cup N''
\end{equation*}
and a similar augment shows that $X(\Delta(B))$ is not Fano.
If $J_1 \cap J_2 \notin B$ and $J_1 \cup J_2=S$,
then we have $|(B|_{J_1 \cap J_2})_{\rm max}| \geq 2$. Let
\begin{equation*}
\tau=\mathbb{R}_{\geq0}(N \cup (B|_{J_1 \cap J_2})_{\rm max}
\cup N' \cup (B|_{(J_1 \triangle J_2) \setminus \{j_1, j_2\}})_{\rm max}).
\end{equation*}
Note that $\tau$ is an $(n-1)$-dimensional cone.
Since $e_{J_1 \cup J_2}=e_S=0$, it follows that
\begin{equation*}
e_{J_1}+e_{J_2}-\sum_{C \in (B|_{J_1 \cap J_2})_{\rm max}}e_C=0.
\end{equation*}
Hence by Proposition \ref{intersectionnumber},
we have $(-K_{X(\Delta(B))}.V(\tau))=2-|(B|_{J_1 \cap J_2})_{\rm max}| \leq 0$.
By Proposition \ref{Fano}, $X(\Delta(B))$ is not Fano.

$(2') \Rightarrow (1')$: Let $N \cup \{I_1\}$ and $N \cup \{I_2\}$
be two maximal nested sets of $B$
with the intersection $N \in \mathcal{N}(B)$.
We need to show that $(-K_{X(\Delta(B))}.V(\mathbb{R}_{\geq0}N))>0$.

{\it The case where $I_1 \cap I_2=\emptyset$}.
By Proposition \ref{pair} (3),
there exist $I_3, \ldots, I_k \in N$ such that
$I_1 \cup I_2, I_3, \ldots, I_k$ are pairwise disjoint
and $I_1 \cup \cdots \cup I_k \in N \cup B_{\rm max}=N \cup \{S\}$. Since
\begin{equation*}
e_{I_1}+e_{I_2}+e_{I_3}+\cdots+e_{I_k}-e_{I_1 \cup \cdots \cup I_k}=0,
\end{equation*}
we have
\begin{equation*}
(-K_{X(\Delta(B))}.V(\mathbb{R}_{\geq0}N))=\left\{\begin{array}{ll}
k-1 & (I_1 \cup \cdots \cup I_k \in N), \\
k & (I_1 \cup \cdots \cup I_k=S). \end{array}\right.
\end{equation*}
Hence $(-K_{X(\Delta(B))}.V(\mathbb{R}_{\geq0}N)) \geq 1$.

{\it The case where $I_1 \cap I_2 \ne \emptyset$}.
By Proposition \ref{pair} (1), we have $I_1 \not\subset I_2$ and $I_2 \not\subset I_1$.
Applying $(2')$ for $I_1$ and $I_2$, we have $I_1 \cup I_2=S$ and $I_1 \cap I_2 \in B$.
Thus $\{I_1 \cap I_2\}=(B|_{I_1 \cap I_2})_{\rm max} \subset N$
by Proposition \ref{pair} (2).
Since $e_{I_1 \cup I_2}=e_S=0$, it follows that
\begin{equation*}
e_{I_1}+e_{I_2}-e_{I_1 \cap I_2}=0.
\end{equation*}
Hence $(-K_{X(\Delta(B))}.V(\mathbb{R}_{\geq0}N))=1$ by Proposition \ref{intersectionnumber}.

Therefore $X(\Delta(B))$ is Fano by Proposition \ref{Fano}.
This completes the proof of Theorem \ref{theorem1}.
\end{proof}

\section{Smooth Fano polytopes associated to finite directed graphs}

We review the construction of an integral convex polytope from a finite directed graph.
Let $G$ be a finite directed graph with no loops and no multiple arrows.
We denote by $V(G)$ and $A(G)$ its node set and arrow set respectively.
$A(G)$ is a subset of $V(G) \times V(G)$. Let $V(G)=\{1, \ldots, n+1\}$.
For $\overrightarrow{e}=(i, j) \in A(G)$,
we define $\rho(\overrightarrow{e})\in \mathbb{R}^{n+1}$ to be $e_i-e_j$.
We define $P_G$ to be the convex hull of $\{\rho(\overrightarrow{e}) \mid
\overrightarrow{e} \in A(G)\}$ in $\mathbb{R}^{n+1}$.
$P_G$ is an integral convex polytope in $H=\{(x_1, \ldots, x_{n+1}) \in \mathbb{R}^{n+1} \mid
x_1+\cdots+x_{n+1}=0\}$.

An integral convex polytope is said to be {\it Fano} if the origin is the only lattice point
in the interior, and it is said to be {\it smooth} if the vertices of every facet
form a basis for the lattice.
Not all finite directed graphs yield smooth Fano polytopes.
See \cite{Higashitani} for the characterization of finite directed graphs
that yield smooth Fano polytopes of dimension $n$.

We state our second main result:

\begin{theorem}\label{theorem2}
Let $B$ be a building set.
If the associated toric variety $X(\Delta(B))$ is Fano,
then there exists a finite directed graph $G$ such that
$P_G$ is a smooth Fano polytope and its associated fan is isomorphic to $\Delta(B)$.
\end{theorem}

For a connected building set $B$ on $S$, we put
\begin{equation}\label{u}
U=\{I \in B \setminus \{S\} \mid \mbox{ there exists } J \in B \setminus \{S\}
\mbox{ s.t. } I \cap J \ne \emptyset \mbox{ and } I \cup J=S\}.
\end{equation}

\begin{lemma}\label{intersection}
Let $B$ be a connected building set on $S$ such that $X(\Delta(B))$ is Fano.
If $I, J \in U$ with $I \ne J$ and $I \cap J \ne \emptyset$,
then we have $I \cup J=S$ and $I \cap J \in B$.
\end{lemma}

\begin{proof}
Let $I, J \in U$ with $I \ne J$ and $I \cap J \ne \emptyset$.
We show that $I \not\subset J$ and $J \not\subset I$.
Assume $I \subsetneq J$ for contradiction.
There exists $K \in U$ such that $I \cap K \ne \emptyset$ and $I \cup K=S$.

First we apply Theorem \ref{theorem1} (2) for $J$ and $K$.
We have $J \cap K \supset I \cap K \ne \emptyset$
and $I \setminus K \subset J \setminus K$.
Let $x \in S \setminus J$. Then we must have $x \notin I$ and thus $x \in K$.
Hence $x \in K \setminus J$.
Thus Theorem \ref{theorem1} (2) implies $J \cap K \in B$.

Next we apply Theorem \ref{theorem1} (2) for $I$ and $J \cap K$.
We have $I \cap (J \cap K)=I \cap K \ne \emptyset$
and $I \setminus K \subset I \setminus (J \cap K)$.
Let $x \in J \setminus I$. Then we must have $x \in K$.
Hence $x \in (J \cap K) \setminus I$.
Thus Theorem \ref{theorem1} (2) implies $I \cup (J \cap K)=S$.
This contradicts that $I \cup (J \cap K)=J \cap (I \cup K)=J \subsetneq S$.
Therefore $I \not\subset J$.

Similarly we have $J \not\subset I$.
Theorem \ref{theorem1} (2) implies $I \cup J=S$ and $I \cap J \in B$.
This completes the proof.
\end{proof}

\begin{lemma}\label{3types}
Let $B$ be a connected building set on $S$ such that $X(\Delta(B))$ is Fano.
Then the following hold:
\begin{enumerate}
\item $U$ in (\ref{u}) must be one of the following:
\begin{enumerate}
\item $U=\emptyset$.
\item $|U|=2$.
\item $U=\{I, J, S \setminus (I \cap J)\}$ for some $I, J \in B$,
and the union of any two elements of $U$ is $S$.
\end{enumerate}
\item Let $I, J \in U$ with $I \cap J \ne \emptyset \mbox{ and } I \cup J=S$.
If $K \in B \setminus \{S, I, J\}$ with $K \not\subset I \setminus J, K \not\subset I \cap J$
and $K \not\subset J \setminus I$,
then we have $K=S \setminus (I \cap J)$.
\end{enumerate}
\end{lemma}

\begin{proof}
(1) If $U \ne \emptyset$, then we have $|U| \geq 2$. Suppose $|U| \geq 3$.
Let $I \in U$.
There exists $J \in U$ such that  $I \cap J \ne \emptyset$ and $I \cup J=S$.
Let $K \in U \setminus \{I, J\}$.
We may assume $I \cap K \ne \emptyset$.
Lemma \ref{intersection} implies $I \cup K=S$ and $I \cap K \in B$.
Since $J \setminus I \subset K$, we have $J \cap K \ne \emptyset$.
Lemma \ref{intersection} implies $J \cup K=S$ and $J \cap K \in B$.

Assume $I \cap J \cap K \ne \emptyset$ for contradiction.
We apply Theorem \ref{theorem1} (2) for $I \cap K$ and $J \cap K$.
We have $I \setminus J \subset (I \cap K) \setminus (J \cap K)$
and $J \setminus I \subset (J \cap K) \setminus (I \cap K)$.
Thus Theorem \ref{theorem1} (2) implies $(I \cap K) \cup (J \cap K)=S$.
This contradicts that $(I \cap K) \cup (J \cap K)=(I \cup J) \cap K=K \subsetneq S$.
Hence $I \cap J \cap K=\emptyset$.
Thus $x \in I \cap J$ implies $x \notin K$.

On the other hand, $x \in S \setminus K$ implies $x \in I \cap J$,
since $I \cup K=J \cup K=S$.
Thus $K=S \setminus (I \cap J)$.
Therefore we must have $U=\{I, J, S \setminus (I \cap J)\}$.
The union of any two elements of $U$ is $S$.

(2) Let $K \in B \setminus \{S, I, J\}$
with $K \not\subset I \setminus J, K \not\subset I \cap J$
and $K \not\subset J \setminus I$.
We may assume that there exists $x \in K \setminus I$.
If $I \cap K=\emptyset$, then we have $K \subset J \setminus I$, which is a contradiction.
Hence $I \cap K \ne \emptyset$.
If $I \subset K$, then $J \cup K \supset I \cup J=S$. Thus $K \in U$.
However (1) implies $K=S \setminus (I \cap J)$, which is a contradiction.
Hence $I \not\subset K$.
Theorem \ref{theorem1} (2) implies $I \cup K=S$.
Thus $K \in U$ and (1) implies $K=S \setminus (I \cap J)$.
\end{proof}

For a building set $B$ on $S$, we put
\begin{equation*}
l(B)=\max\{k \mid \mbox{ there exist } I_1, \ldots, I_k \in B \mbox{ such that }
|I_1| \geq 2, I_1 \subsetneq \cdots \subsetneq I_k\}
\end{equation*}
and we define $m(B)$ to be
\begin{equation*}
\left\{\begin{array}{ll}
\{I \in B \mid |I| \geq 2; \exists I_2, \ldots, I_{l(B)} \in B \mbox{ s.t. }
I \subsetneq I_2 \subsetneq \cdots \subsetneq I_{l(B)}\} & (l(B) \geq 2), \\
\{I \in B \mid |I| \geq 2\} & (l(B) \leq 1). \end{array}\right.
\end{equation*}
$l(B)$ is understood to be zero when $B$ consists of the singletons.

\begin{lemma}\label{numbering}
Let $B$ be a building set on $S$ such that
$I, J \in B$ with $I \cap J \ne \emptyset$ implies $I \subset J$ or $J \subset I$.
Then there exists a bijection $f:S \rightarrow \{1, \ldots, |S|\}$ such that
$f(I)$ is an interval for any $I \in B$, that is,
$f(I)$ is equal to $[i, j]=\{x \in \{1, \ldots, |S|\} \mid i \leq x \leq j\}$ for some
$1 \leq i \leq j \leq |S|$.
\end{lemma}

\begin{proof}
We use induction on $l(B)$.
It is obvious for $l(B)=0$ and for $l(B)=1$. Assume $l(B) \geq 2$.
Note that $B \setminus m(B)$ is a building set and $m(B)$ is pairwise disjoint.
We have $l(B \setminus m(B))=l(B)-1$.
By the hypothesis of induction,
there exists a bijection $f:S \rightarrow \{1, \ldots, |S|\}$ such that
$f(I)$ is an interval for any $I \in B \setminus m(B)$.
Since $m(B \setminus m(B))$ is also pairwise disjoint, for any $I \in m(B)$,
there exists unique $J \in m(B \setminus m(B))$ such that $I \subset J$.
Hence for each $J \in m(B \setminus m(B))$, we can modify $f|_J:J \rightarrow f(J)$
without changing the image of $J$
so that every $f|_J(I)$ is an interval.
Thus we can construct a bijection satisfying the condition.
Therefore the assertion holds for $l(B)$.
\end{proof}

\begin{proof}[Proof of Theorem \ref{theorem2}]
By connecting finite directed graphs that yield toric Fano varieties with one node,
we obtain a graph that yields a toric variety isomorphic to
the product of the toric Fano varieties of the graphs.
Hence it suffices to prove the assertion when $B$ is connected
and $S=\{1, \ldots, n+1\}$.
By Lemma \ref{3types} (1), $U$ in (\ref{u}) falls into the following three cases:

(a) {\it The case where $U=\emptyset$}.
Since $X(\Delta(B))$ is Fano, $B$ satisfies the assumption of Lemma \ref{numbering}.
Hence we may assume that every element of $B$ is an interval.
We define a finite directed graph $G$ as follows: Let $V(G)=\{1, \ldots, n+1\}$.
For $K=[i, j] \in B \setminus \{S\}$, we put
\begin{equation*}
\overrightarrow{e}_K=\left\{\begin{array}{ll}
(i, j+1) & (1 \leq j \leq n), \\
(i, 1) & (j=n+1). \end{array}\right.
\end{equation*}
Let $A(G)=\{\overrightarrow{e}_K \mid K \in B \setminus \{S\}\}$.
We define a linear isomorphism $F:H \rightarrow \mathbb{R}^n$ by
$e_i-e_{i+1} \mapsto e_i$ for $1 \leq i \leq n$.
Then $F$ induces a bijection from
$\{\rho(\overrightarrow{e}_K) \mid K \in B \setminus \{S\}\}$
to $\{e_K \mid K \in B \setminus \{S\}\}$,
which is the set of vertices of the smooth Fano polytope corresponding to $\Delta(B)$.

(b) {\it The case where $|U|=2$}.
Let $U=\{I, J\}$. We may assume that $I=[1, b]$ and $J=[a, n+1]$
for $1<a\leq b<n+1$.
By Lemma \ref{3types} (2), we have
$B=\{S\} \cup U \cup B|_{I \setminus J} \cup B|_{I \cap J} \cup B|_{J \setminus I}$.
Note that $I \setminus J, I \cap J$ and $J \setminus I$ are intervals.
Furthermore, since $X(\Delta(B))$ is Fano,
$B|_{I \setminus J}, B|_{I \cap J}$ and $B|_{J \setminus I}$
satisfy the assumption of Lemma \ref{numbering}.
Hence we may assume that every element of $B$ is an interval.
Let $V(G)=\{1, \ldots, n+1\}$ and
$A(G)=\{\overrightarrow{e}_K \mid K \in B \setminus \{S\}\}$.
Then the isomorphism $F$ induces a bijection from
$\{\rho(\overrightarrow{e}_K) \mid K \in B \setminus \{S\}\}$
to $\{e_K \mid K \in B \setminus \{S\}\}$.

(c) {\it The case where $U=\{I, J, S \setminus (I \cap J)\}$ for some $I, J \in B$,
and the union of any two elements of $U$ is $S$}.
We may assume that $I=[1, b]$ and $J=[a, n+1]$
for $1<a\leq b<n+1$.
By Lemma \ref{3types} (2), we have
$B=\{S\} \cup U \cup B|_{I \setminus J} \cup B|_{I \cap J} \cup B|_{J \setminus I}$.
$I \setminus J, I \cap J$ and $J \setminus I$ are intervals.
Since $X(\Delta(B))$ is Fano,
$B|_{I \setminus J}, B|_{I \cap J}$ and $B|_{J \setminus I}$
satisfy the assumption of Lemma \ref{numbering}.
Hence we may assume that every element of $B \setminus \{S \setminus (I \cap J)\}$
is an interval.
Let $V(G)=\{1, \ldots, n+1\}$ and
$A(G)=\{\overrightarrow{e}_K \mid K \in B \setminus \{S, S \setminus (I \cap J)\}\}
\cup\{(b+1, a)\}$.
Then $F$ induces a bijection from
$\{\rho(\overrightarrow{e}_K) \mid K \in B \setminus \{S, S \setminus (I \cap J)\}\}
\cup \{\rho((b+1, a))\}$
to $\{e_K \mid K \in B \setminus \{S\}\}$.

Therefore we have constructed a finite directed graph $G$
such that the fan associated to $P_G$ is isomorphic to $\Delta(B)$.
This completes the proof of Theorem \ref{theorem2}.
\end{proof}

\begin{example}
Let $S=\{1, 2, 3, 4, 5\}$ and
\begin{equation*}
B=\{\{1\}, \{2\}, \{3\}, \{4\}, \{5\}, \{2, 3\}, \{4, 5\}, \{1, 2, 3\}, \{2, 3, 4, 5\}, \{1, 2, 3, 4, 5\}\}.
\end{equation*}
Then by Theorem \ref{theorem1}, the toric variety $X(\Delta(B))$ is Fano.
We define a finite directed graph $G$ by $V(G)=\{1, 2, 3, 4, 5\}$ and
\begin{equation*}
A(G)=\{(1, 2), (2, 3), (3, 4), (4, 5), (5, 1), (2, 4), (4, 1), (1, 4), (2, 1)\}.
\end{equation*}
Then $\Delta(B)$ is isomorphic to the fan associated to the smooth Fano polytope $P_G$.
\begin{figure}[htbp]
\begin{center}
\includegraphics[width=4cm]{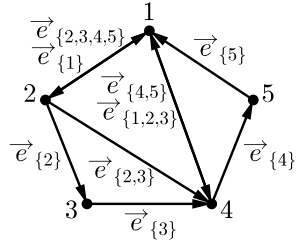}
\caption{the directed graph $G$.}
\label{directedgraph}
\end{center}
\end{figure}
\end{example}

\begin{example}
The converse of Theorem \ref{theorem2} is not true.
The finite directed graphs in Figure \ref{directedgraph2}
yield smooth Fano polytopes (see \cite[Theorem 2.2]{Higashitani}).
However, these polytopes cannot be obtained from building sets.
\begin{figure}[htbp]
\begin{center}
\includegraphics[width=5cm]{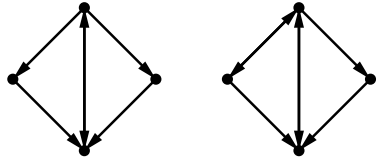}
\caption{directed graphs whose smooth Fano polytopes cannot be obtained from building sets.}
\label{directedgraph2}
\end{center}
\end{figure}
\end{example}

\end{document}